\documentclass[12pt,oneside]{amsart}
\usepackage[english]{babel}
\usepackage{amssymb}
\usepackage{amsthm}
\usepackage{mathtools}

\usepackage[foot]{amsaddr}
\usepackage{amsmath}
\usepackage{a4wide}
\usepackage{nccmath}
\usepackage{esvect}
\usepackage{enumitem}
\usepackage{hyperref}
\usepackage{cleveref}
\usepackage{subfiles}
\hypersetup{
    colorlinks,
    linkcolor=red,
    citecolor=black,
    urlcolor=blue
}
\usepackage{verbatim}
\usepackage[dvips]{graphicx}
\usepackage{t1enc}
\usepackage{color}
\usepackage[margin=0.5in]{geometry}
\usepackage{graphicx}
\usepackage{color,soul}
\usepackage[parfill]{parskip}   
\usepackage[numbers,sort&compress]{natbib}
\usepackage{tikz}
\usepackage{tkz-graph}
\usepackage{amsmath}
\usepackage{caption}
\usepackage{subcaption}
\usepackage{mathtools}

\usetikzlibrary{graphs}
\usetikzlibrary{graphs.standard}

\makeatletter
\def\subsubsection{\@startsection{subsubsection}{3}%
  \z@{.5\linespacing\@plus.7\linespacing}{.1\linespacing}%
  {\normalfont\itshape}}
\makeatother

\newtheorem{thm}{Theorem}[section]
\newtheorem{lem}[thm]{Lemma}
\newtheorem{lemma}[thm]{Lemma}
\newtheorem{conjecture}[thm]{Conjecture}

\newtheorem{prop}[thm]{Proposition}

\newtheorem{definition}[thm]{Definition}

\newtheorem{question}[thm]{Question}

\newtheorem{claim}[thm]{Claim}

\newcommand{\N}{\mathbb N}
\renewcommand{\P}{\mathbb P}

\newcommand{\PP}{\mathbb{P}}

\newcommand{\ignore}[1]{}

\begin{document}

\author{Vojt\u{e}ch Dvo\u{r}\'ak}
\address[Vojt\u{e}ch Dvo\u{r}\'ak]{Department of Pure Maths and Mathematical Statistics, University of Cambridge, UK}
\email[Vojt\u{e}ch Dvo\u{r}\'ak]{vd273@cam.ac.uk}

\author{Rebekah Herrman}
\address[Rebekah Herrman]{Department of Mathematical Sciences, The University of Memphis, Memphis, TN}
\email[Rebekah Herrman]{rherrman@memphis.edu}

\author{Peter van Hintum}
\address[Peter van Hintum]{Department of Pure Maths and Mathematical Statistics, University of Cambridge, UK}
\email[Peter van Hintum]{pllv2@cam.ac.uk}

\title[The Eternal Game Chromatic Number of Random Graphs] 
{The Eternal Game Chromatic Number of Random Graphs}

\linespread{1.3}
\pagestyle{plain}

\begin{abstract}
The eternal graph colouring problem, recently introduced by Klostermeyer and Mendoza \cite{klostermeyer}, is a version of the graph colouring game, where two players take turns properly colouring a graph. In this note, we study the eternal game chromatic number of random graphs. We show that with high probability  $\chi_{g}^{\infty}(G_{n,p}) = (\frac{p}{2} + o(1))n$ for odd $n$, and also for even $n$ when $p=\frac{1}{k}$ for some $k \in \N$. The upper bound applies for even $n$ and any other value of $p$ as well, but we conjecture in this case this upper bound is not sharp.  Finally, we answer a question posed by Klostermeyer and Mendoza.
\end{abstract}
\maketitle 

\section{Introduction}
The \textit{vertex colouring game} was introduced by Brams \cite{gardner} in 1981; it was later rediscovered by Bodlaender \cite{bodlaender}. In this game, two players, Alice and Bob, take turns choosing uncoloured vertices from a graph, $G$, and assigning a colour from a predefined set $\{1,\dots, k\}$, such that the resulting partial colouring of $G$ is proper. Bob wins, if at some stage, he or Alice chooses a vertex that cannot be properly coloured. Alice wins if each chosen vertex can be properly coloured. The game chromatic number $\chi_g(G)$ is the smallest integer $k$ such that if there are $k$ colours, Alice has a winning strategy in the vertex colouring game. This number is well defined, as Alice can win if the number of colours is at least the number of vertices. The vertex colouring game has been well studied \cite{dinski, faigle, guan1999game, zhu1999game}. In particular, Bohman, Frieze and Sudakov \cite{bohman} studied the game chromatic number of random graphs $G_{n,p}$ and found that with high probability, $(1-\epsilon)\frac{n}{\log(pn)}\leq\chi_g(G_{n,p})\leq (2+\epsilon)\frac{n}{\log(pn)}$, where all logarithms have base $\frac{1}{1-p}$.  Keusch and Steger \cite{keusch} improved the result to $\chi_g(G_{n,p})=(1+o(1)) \frac{n}{\log(pn)}$ with high probability, implying $\chi_g(G_{n,p})=(2+o(1))\chi(G_{n,p})$ with high probability by a classic result of Bollob\'as \cite{bollobas1988chromatic}. Both of the results require lower bounds on $p$ decaying with $n$ slowly. Frieze, Haber and Lavrov \cite{Frieze2011OnTG} studied the game on sparse random graphs, finding that for $p=d/n$, $\chi(G_{n,p})=\Theta(\frac{d}{\ln(d)})$, where $d\leq n^{-1/4}$ is at least a large constant.
 
This vertex colouring game requires Alice and Bob to colour the vertices once, attaching no value to the colouring that is produced at the end of the round. In a variant of the game called the \emph{eternal vertex colouring game}  recently introduced by Klostermeyer and Mendoza \cite{klostermeyer}, the focus is shifted by continuing the game after a colouring is produced. 

In the eternal vertex colouring game, there is a fixed set of colours $\{1,\dots, k\}$. The game consists of rounds, such that in each round, every vertex is coloured exactly once. The first round proceeds precisely the same as the vertex colouring game, with Alice taking the first turn. During all further rounds, players keep choosing vertices alternately.  After choosing a vertex, the player assigns a colour to the vertex which is distinct from its current colour such that the resulting colouring is proper. Each vertex retains its colour between rounds until it is recoloured. Bob wins if at any point the chosen vertex does not have a legal recolouring, while Alice wins if the game is continued indefinitely. The eternal game chromatic number $\chi_g^{\infty}(G)$ is the smallest number $k$ such that Alice has a winning strategy. Note that if $k\geq\Delta(G)+2$, there will always be a colour available for every vertex, so $\chi_g^\infty(G)$ is well-defined.

As Alice and Bob alternate their turns, the parity of the order of the graph determines whose turn it is at the beginning of the second round. For even order, Alice always has the first move, while for odd order Bob gets to play first in all even rounds.

This game has not been well studied, but Klostermeyer and Mendoza \cite{klostermeyer} obtained some basic results pertaining to paths, cycles, and balanced bipartite graphs. 

In this paper, we determine $\chi_{g}^{\infty}(G_{n,p})$ for 
$n$ odd by putting together the following two results.

\begin{thm} For all $p\in (0,1)$ constant, for odd $n$, with high probability,
$$\chi_{g}^{\infty}(G_{n,p}) = (1+o(1))\frac{pn}{2}.$$
\end{thm}

\begin{thm}
For all $p\in (0,1)$ constant, for even $n$,  
with high probability, $$\chi_{g}^{\infty}(G_{n,p}) \leq (1+o(1))\frac{pn}{2}.$$
Moreover, when $p=\frac{1}{l}$ for some $l \in \N$, $$\chi_{g}^{\infty}(G_{n,p}) = (1+o(1))\frac{pn}{2}.$$
\end{thm}



The difference in the even and odd cases is because when $n$ is odd, Bob moves first in the second round. Also, note that we made no efforts to optimize $o(n)$ terms. For the unresolved case when $n$ is even and $p\not\in\{\frac{1}{2},\frac{1}{3},\dots\}$, we conjecture the following.
\begin{conjecture}\label{conjecture}
$\forall p\in(0,1)\setminus \{\frac{1}{2},\frac{1}{3},\dots\},\exists \epsilon>0$ such that with high probability, $$\chi_g^{\infty}(G_{n,p}) \leq (1-\epsilon) \frac{pn}{2}.$$
\end{conjecture}

The structure of the paper is as follows. In Section 2, we prove the upper bound for $\chi_g^{\infty}(G_{n,p})$. In this proof, we make no distinction between odd and even values of $n$. In Section 3, we prove the corresponding lower bound for odd $n$. In Section 4, we prove a generalization of the result in Section 3, which we then use in Section 5 to get the lower bound for the case $p=1/l$ for some $l\in\mathbb{N}$. Along the way, we use various structural results about the random graph $G_{n,p}$. As the proofs of these are usually quite easy but technical, we collect all of them in Appendix A. Finally in Section 6, we provide answer to one of the questions posed in the paper of Klostermeyer and Mendoza.

Throughout the paper we will use the following conventions. We say that a result holds in $G_{n,p}$ with high probability (\emph{whp}) if the probability that it holds tends to $1$ as $n \rightarrow \infty$. The neighbourhood of a vertex $v$ in a graph, $G$, denoted $N(v)$, will be the vertices of $G$ that $v$ is connected to, as well as $v$ itself. The partition of a set $X$ will refer to a collection of disjoint non-empty subsets whose union is $X$.



\section{Upper bound}

In this section, we show the following proposition.
\begin{prop}
For any fixed $p,\epsilon\in(0,1)$, whp $\chi_{g}^{\infty}(G_{n,p}) \leq (\frac{p}{2} + \epsilon)n $.
\end{prop}

To prove this, we formulate a deterministic strategy for Alice and prove that whp this strategy enables her to prevent Bob from winning when the game is played with $ (\frac{p}{2} + \epsilon)n $ colours. 

The biggest danger facing Alice is that at the end of some round, Bob would manage to introduce all the colours in the neighbourhood of at least one vertex. He could then win by choosing one of those vertices at the beginning of the next round. Thus, her strategy should be to ensure that at any point of each round, she has coloured roughly as many vertices in the neighbourhood of any single vertex as Bob has, and she should use few colours on them. If, at some point during a round, a vertex has many colours in its neighbourhood compared to other vertices, Alice might be forced to colour it so Bob cannot win by choosing it later that same round. Fortunately for Alice, the number of times she is forced to colour a vertex with many different coloured neighbours is so few that she can still follow her strategy.

Consider the following four properties of a graph $G_{n,p}$.

\begin{enumerate}[label=(\roman*)]
 \item\label{vertexcondn}
Every vertex of $G_{n,p}$ has degree at most $(p+\frac{\epsilon}{100})n $

\item\label{fewdangers}
 There exists a constant $K=K(\epsilon,p)$ such that $G_{n,p}$ does not contain sets $A,B,S\subset V(G)$, with $|A\cap B|=0$, $|A|=|B|\geq\frac{ \epsilon}{200} n, |S|=K$, such that every $v \in S$ is connected to at least $\frac{\epsilon}{200} n$ more vertices in $B$ than in $A$.
\item\label{fewcolours}
There exist constants $\beta=\beta(\epsilon,p)$, $\frac{\epsilon}{100}>\beta>0$ and $C=C(\epsilon,p)$ such that the following holds: for any colouring of $G_{n,p}$ by $(\frac{p}{2} + \epsilon) n$ colours, the number of vertices that have all but at most $2\beta n$ colours in their neighbourhood is at most $C \log n$.
 \item\label{othercolours}
 There exist constants $\gamma=\gamma(\epsilon,p)>0$ and $D=D(\epsilon,p)$ such that, in any colouring of $G_{n,p}$ by $(\frac{p}{2} + \epsilon) n$ colours, the number of vertices that have all but at most $\gamma n$ of the colours $1,2,...,\frac{\epsilon}{200}n$  in their neighbourhood is at most $D\log n$.
 \end{enumerate} 
 
We prove in \Cref{structure} that each of these holds whp in $G_{n,p}$. For the remainder of this paragraph, we will assume \ref{vertexcondn} through \ref{othercolours} hold for the graph $G_{n,p}$. Note that as we are assuming finitely many properties, each of which holds whp, then whp all of them hold simultaneously.

For a particular round of the game, let $A_i$ and $B_i$ denote the sets of vertices played by Alice and Bob respectively in the first $i$ moves of that round. We shall define the vertices that threaten Alice's chance of winning as dangerous. 

\begin{definition}
For a fixed round of play, let $D_i$ denote the set of \emph{dangerous vertices at $i$ moves}, denoted $D_i$. A vertex $v$ belongs to $D_i$ if for some number of moves $j\leq i$, Bob has played at least $\frac{\epsilon}{100}n$ times more in the neighbourhood of $v$ than Alice has, i.e. $|N(v)\cap B_j|\geq |N(v)\cap A_j|+\frac{\epsilon}{100}n$.
\end{definition}
We additionally define vertices that Alice can colour to maintain some symmetry in the game as follows.
\begin{definition}
Let $S$ be a finite subset of vertices of a graph, $G$. For a vertex $v\not\in S$, we say that a vertex $w$ \emph{mirrors $v$ with respect to $S$} if $w\not\in S$ and for any $t \in S$, $G$ contains an edge $vt$ if and only if it contains an edge $wt$.
\end{definition}

Let $C= \{1,2,...,(\frac{p}{2}+ \epsilon)n\}$ be the set of colours used in the game. We call a colour \emph{large} if it is at least $\frac{\epsilon}{200} n$, and \emph{small} otherwise.

Alice will use the following strategy at the $i^{th}$ move of a round: from the list below, she chooses the first point that applies, and colours the corresponding vertex with smallest colour available to that vertex. If there are multiple vertices for the same point on the list, she chooses one of these arbitrarily.
\begin{enumerate}
\item If there is a vertex $v$ that misses less than $ \beta n $ colours in its neighborhood and such that $v$ has not yet been coloured in the current round, she chooses $v$.
\item If Bob played a vertex $w$ for his previous move, $w$ is not dangerous, and there is a vertex $v$ which mirrors $w$ with respect to $D_i$, she chooses $v$.
\item She chooses an arbitrary vertex.
\end{enumerate}
We shall prove that because \ref{vertexcondn}, \ref{fewdangers}, \ref{fewcolours}, and \ref{othercolours} hold, selecting vertices in order of priority will ensure that Bob can never win the eternal vertex colouring game for sufficiently large $n$. To show Bob cannot win, we prove the following lemma. 




\begin{lemma}\label{nth}
For $n$ sufficiently large, at the beginning of the $k^{th}$ round of play for $k\geq 2$, the following two conditions hold:

$\bullet$ During the $(k-1)^{st}$ round, there was no vertex $v$ such that number of times Bob played in neighbourhood of $v$ was more than $\frac{\epsilon}{50} n$ greater than number of times Alice played in neighbourhood of $v$.

$\bullet$ Alice used no more than $\frac{\epsilon}{100} n$ colours in $(k-1)^{st}$ round.

Then, by playing according to the above described strategy, Alice ensures the following:

$\bullet$ Bob does not win during $k^{th}$ round

$\bullet$ During $k^{th}$ round, there is no vertex $v$ such that number of times Bob played in neighbourhood of $v$ is more than  $\frac{\epsilon}{50} n$ greater than number of times Alice played in neighbourhood of $v$.

$\bullet$ Alice uses no more than $\frac{\epsilon}{100} n$ colours in the $k^{th}$ round.
\end{lemma}

Note that Lemma \ref{nth} implies that, if Alice plays according to the strategy described above,  Bob can never win the eternal graph colouring game. 

\begin{proof}
The first step is to establish that at beginning of the round, each vertex misses more than $2\beta(\epsilon,p)n$ colours in its neighbourhood, so that there is no immediate threat to Alice. In the first round, this is immediate, as no colour is used yet. When $k\geq2$,  Alice uses at most $\frac{\epsilon}{100} n$ colours in the neighbourhood of any vertex $v$. Bob played at most $\frac{\epsilon}{50} n$ more moves in the  neighbourhood of $v$ than Alice did, so by property \ref{vertexcondn}, Bob played at most $(\frac{p}{2}+\frac{\epsilon}{200}+\frac{\epsilon}{100})n$ colours in the neighborhood of $v$. Hence, at least $\frac{39\epsilon}{40}n>2\beta(\epsilon,p) n$ colours are missing from the neighbourhood of  $v$.

Now, if some vertex misses at most $\beta n$ colours at any point during the round, then in particular at least one of the times $\beta n, 2 \beta n,...,\lfloor \beta^{-1}\rfloor \beta n$, this vertex missed at most $2 \beta n$ colours. By property \ref{fewcolours}, we conclude there are at most $C \beta^{-1} \log n = o(n)$ vertices that, at some point in this round, have missed at most $\beta n$ colours. Recall that colouring vertices that miss at most $\beta n$ colours is of the highest priority in Alice' strategy. If Bob were to create a vertex seeing all colours that was not yet played in this round, Alice must have spend the previous $\beta n$ moves playing in other vertices missing at most $\beta n$ colours in their neighbourhoods. However, this contradicts the fact that there were at most  $C \beta^{-1} \log n <\beta n$ such vertices. Hence, Alice can colour all such vertices in time.

Next, note that Alice uses $o(n)$ (and in fact only constantly many) moves that are arbitrary.  If Alice colours an arbitrary vertex, then either Bob played a dangerous vertex for his previous move or she cannot mirror Bob on the current set of dangerous vertices. By property \ref{fewdangers}, there are at most $K$ vertices declared dangerous during the round, so Bob can play in a dangerous vertex no more than $K$ times. On the other hand, consider if Bob did not play dangerous vertex and Alice cannot mirror his move on $D$, the set of dangerous vertices. If we partition the rest of the graph into $2^{|D|}\leq 2^{K}$ classes according to which vertices of $D$ they are connected to, Bob must have just played last vertex from one of these classes. Hence, Alice must have played at most $K+2^K = O(1)$ arbitrary moves in any particular round.  

Following this strategy, Alice also ensures that Bob will play in the neighbourhood of any vertex at most $\frac{\epsilon}{50} n$ more than Alice does. Indeed, once Bob has played $\frac{\epsilon}{100} n$ more colours in the neighbourhood of any vertex, $v$, it is declared dangerous. She then plays in neighbourhood of $v$ whenever Bob does, except $o(n)$ times when she plays a move of type 1 or an arbitrary vertex.

Finally, note that Alice uses large colours only if the vertex she wants to colour is connected to all the small colours. If at any point during the round a vertex is connected to all the small colours, then at least one of the times $\gamma n, 2 \gamma n,..,\lfloor \gamma^{-1}\rfloor \gamma n$, this vertex must have been missing at most $\gamma n$ colours. By property \ref{othercolours}, there could have only been at most $D \gamma^{-1} \log n = o(n)$ vertices that were connected to all small colours at some point in this round. Hence, as she can colour all other vertices with small colours, Alice uses at most $\frac{\epsilon}{200}n + D \gamma^{-1} \log n < \frac{\epsilon}{100} n$ colours during the round.
\end{proof}

\section{Lower bound for odd n}

In this section, we prove the lower bound for the eternal game chromatic number on a graph with an odd number of vertices.
\begin{prop}\label{oddprop}
For any $p,\epsilon\in(0,1)$ fixed, whp $\chi_{g}^{\infty}(G_{2m+1,p}) \geq (\frac{p}{2} - \epsilon)(2m+1) $. 
\end{prop}
For convenience, we shall let $n=2m+1$. Fix any vertex $v$ of the graph $G_{n,p}$.

We will show that whp, Bob can ensure that in the first round, all $(\frac{p}{2} - \epsilon)n$ colours are in the closed neighbourhood of $v$ in $G_{n,p}$. Bob then wins in the first move of the second round, by choosing $v$.

Bob can introduce all colours in $N(v)$ by playing in $N(v)$ whenever Alice does, thus ensuring he plays in at least roughly half of the vertices in $N(v)$ while introducing a new colour every time. Some set of vertices $X$ outside $N(v)$ might at some point be adjacent to all unplayed vertices of $N(v)$. If Alice were to play some colour not appearing in $N(v)$ in all these vertices, this colour could no longer be introduced to $N(v)$. Fortunately for Bob, the number of such sets will be very limited, and thus Bob can take care of them in time.

Consider following two properties of a random graph.

\begin{enumerate}[label=(\roman*)]

\item\label{nodegreesmall}
Whp, every vertex of $G_{n,p}$ has degree at least $(p-\frac{\epsilon}{100})n$.

\item\label{fewblocks}
For all $\epsilon>0$, and $p\in(0,1)$, there exist positive constants $\delta=\delta(\epsilon,p), K=K(\epsilon,p)$, such that in $G_{n,p}$
\begin{itemize}
\item Whp, there exist no 3 vertices $u,v,w$ such that the number of vertices in the neighbourhood of $u$, but not in the neighbourhood of $v$ or $w$ is at most $\delta n$
\item Whp, for any set $S$ of size $\frac{\epsilon}{100} n$ in the graph, there exist at most $K$ mutually disjoint pairs of vertices $\{a_{i},b_{i}\}$ such that at most $\delta n$ vertices of $S$ are not in $N(a_i)\cup N(b_i)$
\end{itemize}
\end{enumerate} 

Henceforth, we assume our graph has both properties, and fix $\delta=\delta(\epsilon,p)$ and $K=K(\epsilon,p)$. In \Cref{structure}, we show that indeed whp $G_{n,p}$ has these properties.


Note that if at some stage there exists a colour $c$ that does not appear in $N(v)$ and all vertices not yet played in $N(v)$ are adjacent to a vertex of colour $c$, then $c$ will not appear in $N(v)$, which is contrary to Bob's goal.

We introduce the ideas of a \emph{double block} and being \emph{$\alpha$ away from becoming a double block} in order to describe a strategy Bob should take to achieve his goal of filling the neighbourhood of a vertex with several colours.
\begin{definition}
A pair of vertices $a$ and $b$ is called a \emph{double block} if at some stage in the round, all uncoloured vertices in $N(v)$ are in the neighbourhood of either $a$ or $b$ and neither $a$ or $b$ (if coloured) is coloured with a colour appearing in $N(v)$.
\end{definition}

\begin{definition}
A pair of vertices $a$ and $b$ is said to be \emph{$\alpha$} away from becoming a double block, if all but at most $\alpha$ of the uncoloured vertices in $N(v)$ are in the neighbourhood of either $a$ or $b$ and neither $a$ or $b$ (if coloured) is coloured with a colour appearing in $N(v)$.  
\end{definition}

Bob will play according to the following strategy. From the list below, he picks the highest point that applies.
\begin{enumerate}
\item If there exists a colour that appears at least twice outside of $N(v)$ but not in $N(v)$, then Bob plays it in $N(v)$ if it is a valid colouring.

\item If at least $10K$ colours appear nowhere in the graph and there are at least $\frac{\epsilon}{100} n$ uncoloured vertices in the neighbourhood of $v$, then Bob does the blocking moves in the chronological order they were called for, if legally possible. Blocking moves are called for if a pair $\{a,b\}$ of unplayed vertices is $\delta n$ away from becoming a double block. 
Blocking moves consist of the following steps. First colour vertex $a$ a colour not yet appearing in the graph, say $c_a$. Second, unless Alice plays in vertex $b$ or introduces $c_a$ in $N(v)$, play $c_a$ in $N(v)$ and repeat for vertex $b$. If Alice plays in vertex $b$ with a colour $c_b$ not appearing in $N(v)$, play $c_b$ in $N(v)$, and finish by playing $c_a$ in $N(v)$ on the next move.


\item If legally possible, Bob introduces colours appearing once outside of $N(v)$ into $N(v)$, in the order in which they were introduced to the game. 

\item If legally possible, Bob introduces new colours to $N(v)$.
\item Otherwise Bob does anything.
\end{enumerate}

Note that (2) might involve up to four moves for any pair close to becoming a double block. If in between these four moves a situation as in (1) arises, situation (1) takes priority.

\begin{claim}\label{few2moves}
There are no more than $4K$ moves of type (2) used in the first round.
\end{claim}

\begin{proof}
Let $U$ denote the set of vertices that are uncoloured in $N(v)$ when the last blocking moves were played. By the  definition of type (2) moves, $|U| \geq \frac{\epsilon}{100} n$, so property \ref{fewblocks} gives the result.
\end{proof}


Let $T$ be the first move after which precisely $10K-1$ colours are missing in $G_{n,p}$ during the first round. We collect the following observations about $T$:

$\bullet$ $T$ exists and at $T$, at least $\epsilon n$ vertices in $N(v)$ are uncoloured

We shall show that $T\leq (p - 2\epsilon) n$. After Bob's first $(\frac{p}{2} - \epsilon) n$ moves, Alice has also played  $(\frac{p}{2} - \epsilon) n$ moves. As $|N(v)| \geq (p-\frac{\epsilon}{100}) n$, at this stage at least $\frac{199}{100} \epsilon n$ vertices in $N(v)$ are uncoloured. Bob spent at most $4K$ moves playing according to (2), and when he did not, he always introduced a new colour in $N(v)$, if he legally could. Note that if there were still colours missing from $G_{n,p}$ and there were uncoloured vertices in $N(v)$, moves of type (4) were always legal. Hence, unless all colours appear in the graph, Bob played at least $(\frac{p}{2}-\epsilon)n-4K$ colours in $N(v)$ and hence in $G_{n,p}$.

$\bullet$ At $T$, at most $18K$ colours are missing in $N(v)$

Between two consecutive moves of Bob before T, the number of colours appearing outside of $N(v)$ but not in $N(v)$ can increase by at most 2. In fact it only increases if Bob makes a move of type (2). Hence, there are at most $8K$ such colours at time $T$, and result follows.

$\bullet$ At $T$, Bob has played at most $8K$ (1) moves.

Let $c$ be the number of colours appearing outside $N(v)$ and not in $N(v)$. Note that between two consecutive moves of Bob up to time $T$, $c$ increases only if Bob plays a (2) move, in which case it increases by at most 2. On the other hand, note that Bob only plays (1) moves directly after Alice plays a colour already appearing outside $N(v)$. Hence, $c$ decreases whenever Bob plays a (1) move. By \ref{few2moves}, there were at most $4K$ (2) moves, so there were at most $8K$ (1) moves.

$\bullet$ No pair of vertices is closer than $\frac{\delta}{2} n$ to becoming a double-block at any point up to $T$

We know from \ref{fewblocks} that at the beginning of the game, no pair is closer than $\delta n$ to becoming a double block. 
Up to time $T$, whenever a pair gets closer than $\delta n$ to becoming a double block, no (3)-(5) moves are played until this pair is eliminated. However, there are at most $12K$ (1) and (2) moves played until $T$. Hence, no pair can be closer than $\delta n-12K\geq \frac{\delta}{2}n$ to becoming a double-block up to $T$.

$\bullet$ Every colour that does not appear in $N(v)$ at $T$ appears at most once in $G_{n,p}$

Note that the only time before $T$ that there is a colour appearing twice outside $N(v)$, but not inside, is directly after Alice has played this colour. In response, Bob immediately plays that same colour in $N(v)$, which is possible as no pair of vertices is a double block. Hence, if Bob made the last move before $T$, the statement follows. If the last move before $T$ was by Alice, she must have introduced a new colour into the graph by the definition of $T$, which again implies the statement.

Next we claim:

\begin{claim}
In the $18K$ moves of Bob following $T$, he will introduce all colours in $N(v)$.
\end{claim}

\begin{proof}
Moves of type (2) are no longer played after $T$ by their definition. In the next $36K$ moves, $18K$ of which are made by Bob and the $18K$ by Alice, no complete double-block can be created, as all are at least $\frac{\delta}{2} n>36K$ moves away. Since at $T$ at least $\epsilon n$ vertices of $N(v)$ are still uncoloured, during the next $36K<\epsilon n$ moves, there are ample uncoloured vertices in $N(v)$. Hence, Bob can and will introduce a new colour to $N(v)$ every move until all colours appear there, as he ceases the (2) moves.
\end{proof}

Thus we see that Bob will ensure that all colours appear in $N(v)$ during the first round and he will win in the first move of the second round by picking $v$.

\section{Generalization of the lower bound for odd n}
\Cref{oddprop} doesn't trivially extend to even $n$, as it is not enough for Bob to let all colours appear in the neighbourhood of a fixed vertex because Alice could use her first move in the second round to remove one of the colours from this neighbourhood. 

If Bob can manage to play all colours in the neighbourhood of two vertices, with no colour appearing uniquely in the intersection of the neighbourhoods, then Alice can not. This limits how Bob can colour the intersection of two neighbourhoods of vertices. By increasing the number of vertices that simultaneously see all colours, the size of this intersection can be reduced. Our aim is to show that if $p=1/k$ for some $k\in\mathbb{N}$, then for any fixed $l$, Bob can choose $l$ vertices and play all of $(p/2-p^l/2-\epsilon)n$ colours in the neighbourhoods of these vertices. The $p^l/2n$ correction term comes from the intersection of the neighbourhoods of these $l$ vertices. As $l$ can be taken arbitrarily large, this shows $\chi_g^\infty(G_n,p)$ for even $n$ and $p=1/k$.

In this section, we prove a generalization of \Cref{oddprop}, showing that if $V(G)$ is partitioned into constantly many parts and each of the parts is assigned a set of colours of size roughly half the size of the part, Bob can guarantee all these colours to appear in the parts by the end of the first round. This generalizes the notion that Bob could achieve this in the single set $N(v)$. In the next section, we will fix some set of vertices $X$ of constant size and induce partition $\{A_I:I\subset [l]\}$, where $A_I=\{v\in V: N(v)\cap X=I\}$. We show in \Cref{1/klem}, that for the special case $p=\frac{1}{k}$, there exists an appropriate way of assigning colours to the $A_I$'s such that each vertex in $X$ will see all colours after the first round.

\begin{prop}\label{genbob}
$\forall \epsilon,\eta, \gamma>0, l\in\mathbb{N},$ and $p\in(0,1)$, if $X_i\subset V$  $(i\in[l])$ are disjointly chosen independent sets of vertices of the graph $G_{n,p}$ with $|X_i|\geq \gamma n$ and $Y_i\subset [(p/2-\epsilon)n]$ with $|Y_i| \leq \frac{(1-\eta) |X_i|}{2}$, then whp Bob can guarantee that at the the end of round all of the colours in $Y_i$ appear in $X_i$.
\end{prop}

To prove \Cref{genbob} we will use the following generalization of the structural result in \Cref{fewblocks}.

\begin{lemma}\label{generalfewblocks}
For all $m\in\mathbb{N}$, $\alpha>0$ and $p\in(0,1)$, there exist positive constants $\delta=\delta(\alpha,p,m), K=K(\alpha,p,m)$, such that
\begin{itemize}
\item For any set $S$ of size at least $\alpha n$ in the graph, whp there exist no $m$-sets of vertices $\{a_{1},\dots a_{m}\}$ such that at most $\delta n$ vertices of $S$ are not in $\bigcup_{j}N(a_{j})$
\item Whp, for any set $S$ of size at least $\alpha n$ in the graph, there exist at most $K$ mutually disjoint $m$-sets of vertices $\{a_{i,1},a_{i,2},\dots a_{i,m}\}$ such that at most $\delta n$ vertices of $S$ are not in $\bigcup_{j}N(a_{i,j})$
\end{itemize}
\end{lemma}

In order to prove Propositon \ref{genbob}, we will define the concepts of an \emph{end stage}, \emph{m-block} and \emph{$\alpha$ away from becoming an m-block}.

\begin{definition}
Let $T_i$ be the first move after which at most $10K$ of the colours in $Y_i$ are missing from $X_i$, if this exists. After $T_i$, say $X_i$ is in its \emph{end stage}.
\end{definition}

\begin{definition}
Given disjoint sets $X_i\subset V$, at some stage of the round we say a set $\{a_1,\dots,a_{m}\}$ is an \emph{$m$-block} if for some $i$, $X_i$ is not in its end stage, every uncoloured vertex in $X_i$ is in the neighbourhood of some $a_j$, and no $a_j$ is coloured in some colour also appearing in $X_i$.
\end{definition}

\begin{definition}
Given disjoint sets $X_i\subset V$, at some stage of the round we say a set $\{a_1,\dots,a_{m}\}$ is \emph{$\alpha$ away from becoming a $m$-block} if for some $i$, $X_i$ is not in its end stage, all but $\alpha$ of the uncoloured vertices in $X_i$ is in the neighbourhood of some $a_j$, and no $a_j$ is coloured in some colour also appearing in $X_i$.
\end{definition}

\begin{proof}[Proof of \Cref{genbob}]
Let $C_l= \frac{12l}{\eta \gamma}+4$ and let $\delta=\delta(\eta/4,p,100C_ll)$ and $K=K(\eta/4,p,100C_ll)$ as in  \Cref{generalfewblocks}. Bob will play the move of the highest priority that he legally can according to the following list: 
\begin{enumerate}
\item If for some $q\in[l]$, some colour $c$ appears $C_lq$ times in the graph, but it is missing from more than $l-q$ of the $X_i$'s for which $c\in Y_i$, Bob plays it in any of $X_i$'s where it does not yet appear.

\item If for some $i$, $X_i$ is in its end stage, Bob plays the missing colours into it, copying the colour Alice played if it was missing.

\item If there is $100C_l l$ block closer than $\delta n$ moves away from becoming an m-block and at least $\eta/4 n$ vertices in the corresponding $X_i$ are uncoloured, Bob kills it. By killing it, we mean the following sequence of moves. Colour the first vertex of our $100C_l l$-set by some colour that appears less than $C_l q$ times in the graph and is missing from at most $l-q$ of the relevant $X_i$'s. Then make sure in the next moves that this colour also appears in all of its designated $X_i$'s. Repeat this procedure for all vertices of our $100C_l l$-set.

\item If for some $q\in[l]$, some colour appears $C_lq$ times in the graph, but it is still missing from more than $l-q+1$ of $X_i$'s, Bob plays it in any of $X_i$'s where it does not yet appear.

\item Bob plays any colour in $Y_i$ not yet used in  $X_i$ to that $X_i$, if possible in the same $X_i$ as Alice played in the previous move.

\item Bob plays anything anywhere.
\end{enumerate}

Note that (3) might involve up to $100C_ll(l+1)$ moves. If in between these moves a situation as in (1) or (2) arises, those are resolved first.

\begin{claim}\label{fewbigmoves}
Let $C=100C_ll^2(l+1)$. There were no more than $C K$ (3) moves  called for.
\end{claim}

\begin{proof}
Let $U\subset X_i$ denote the set of vertices that are still uncoloured in $X_i$ when the last blocking moves were called for, for this $X_i$. \Cref{generalfewblocks} says $X_i$ called for at most $100C_ll(l+1)K$ (3) moves. Hence, in total at most $100C_ll^2(l+1)K=CK$ (3) moves are called for.
\end{proof}

\begin{claim}\label{fewimportant}
There were no more than $\frac{2l}{C_l}n$ moves of types (1),(2),(3) and (4) during the first round of the game.
\end{claim}

\begin{proof}
Note that at most $\frac{n}{C_l}$ colours appear at least $C_l$ times in the graph. Moreover, these colours prompt a (1) or (4) move at most $l$ times. Finally, there are at most $10Kl$ (2) moves. Hence, there are at most $\frac{l}{C_l}n+10Kl+CK\leq \frac{2l}{C_l}n$, given $n\geq \frac{C_lK}{l}(10l+C)$.
\end{proof}


 We collect the following observations about $T_i$:

$\bullet$ $T_i$ exists and, at $T_i$, at least $\frac{\eta}{4} n$ vertices in $X_i$ are still uncoloured

At the end of round one there were $|X_i|$ moves in $X_i$. Moreover, by \Cref{fewimportant} there were at most $\frac{2l}{C_l}n$ (1)-(4) moves.  
After Bob's first $|Y_i|+\frac{2 l }{C_l} n$ in $X_i$, he has played at most $\frac{2l}{C_l}n$ (1)-(4) moves. He also played in $X_i$ after every move of Alice in that set, except the times when he played (1)-(4) moves. Thus, at least $\eta |X_i|-\frac{6l}{C_l}n\geq \frac{\eta \gamma}{2} n$ of the vertices in $X_i$ are uncoloured. As $C_l\geq \frac{12l}{\eta \gamma}$, this gives the result.


$\bullet$ Let $C'=C+10l$. At $T_i$, Bob has played at most $C'K$ (1) moves. 

For a colour $j$, let $q_j$ be the number such that colour $j$ is missing from $l-q_j$ of its designated sets. Let $r_j$ be the number of times $j$ appears in the graph. If $r_j-q_jC_l>0$, then Bob is forced to play a (1) move. If $r_j-(q_j-1)C_l>0$, then this induces a (4) move. Let $D=\sum_j \max\{r_j-(q_j-1)C_l,0\}$. Note that if $D>0$, then Bob must play a (1),(2),(3) or (4) move. If $D$ increases between consecutive moves of Bob, he must have played a (2) or (3) move. Moreover, $D$ increases by at most 2 in that case. On the other hand, if Bob is prompted to play a (1) move, $D$ decreases by at least $C_l-1>2$. Hence, there are at most as many (1) moves as there are (2) and (3) moves, i.e. at most $CK+10Kl$ (1) moves.



$\bullet$ No pair of vertices is closer than $ \frac{\delta}{2}n$ to becoming a $100C_ll$ block at any point up to $T_i$

By \Cref{generalfewblocks}, at the beginning of the game no $100C_ll$-set of vertices is closer than $\delta n$ to becoming a $100C_ll$ block. 
Whenever, up to time $T_i$, a $100C_ll$-set gets closer than $\delta n$ to becoming a $100C_ll$ block, no (4)-(6) moves are played until this pair is eliminated. However, there are at most $(2C+10l)K$ (1) and (3) moves played until $T_i$. Hence, no $100C_ll$-set gets closer than $\delta n-(2C+10l)K\geq \frac{\delta}{2}n$ to becoming a $100C_ll$ block up to $T_i$.

$\bullet$ Every designated colour that does not appear in $X_i$ at $T_i$ appears at most $C_l l+2$ times in our graph

By the definition of (1) moves, some colour $c$ can never appear $C_l l+2$ times in our graph, yet not appear in some of $X_i$'s with $c\in Y_i$.

Next we claim:

\begin{claim}
In the $10Kl+C'K$ moves of Bob following $T_i$, he will introduce all colours in $X_i$.
\end{claim}

\begin{proof}
Note that while there are still colours missing from $X_i$ in its end stage, Bob only plays (1) and (2) moves, both of which copy the colour Alice played. Hence, the colours missing from $X_i$ can be played at most $2l$ times before being played into $X_i$. At that stage, the colour is played at most $C_ll+2+2l<100C_ll$ times and no $100C_ll$-set is closer than $\frac{\delta}{2}n$ to becoming a $100C_ll$ block, so no $100C_ll$ block will be formed in the endstage of $X_i$. Hence, we can still play this colour in $X_i$.
As we can introduce all the missing colours and we play at most $C'K$ (1) moves, we need at most $10Kl+C'K$ moves to introduce them all.
\end{proof}
Thus, since Bob can introduce all colours into $X_i$ during the end game, the proof of \Cref{genbob} is complete.
\end{proof}

Having proven \Cref{genbob}, we are ready to look at even $n$. 

\section{Even n}
In this section, we shall prove that for particular values of $p$, we can achieve the same lower bound for even $n$ as for odd $n$.

\begin{prop}\label{bigproposition}
Let $p=1/k$ for some $k\in\mathbb{N}$, and $\epsilon>0$. Then whp $\chi_g^\infty(G_{2m,p})\geq (p/2-\epsilon)2m$.
\end{prop}

For convenience write $n=2m$. For given $p,\epsilon>0$, fix $l\in\mathbb{N}$, such that $p^l<\epsilon/100$.

\begin{lem}\label{1/klem}
Let $X\subset V(G)$ be a set of $l$ vertices and $p=1/k$ for some $k\in\mathbb{N}$. There exists  $\eta>0$, and a function $f:\mathcal{P}(X)\mapsto \mathcal{P}((p/2-p^l/2-\epsilon)n)$, assigning to every subset $X'\subset X$, $p^{|X'|}(1-p)^{l-|X'|}(1-\eta)\frac{n}{2}$ colours, such that $\bigcup_{X':x\in X'\subset X} f(X')=[(p/2-p^l/2-\epsilon)n]$ for every $x\in X$.
\end{lem}
To prove this lemma we will use the following auxiliary lemma.

Let $\mathcal{B}(X)$ be the set of all partitions of the set $X$.

\begin{lem}\label{1/kauxlem}
Consider any $ k \in \mathbb{N}$. Let $p=1/k$ and $|X|=l$, then there exists $g:\mathcal {B}(X)\to [0,1]$,  such that for all $A\subset X$;
$$\sum_{T: A\in T\in \mathcal{B}(X)} g(T)= p^{|A|}(1-p)^{l-|A|}$$
\end{lem}
\begin{proof}

Define $g$ as 
$$g(T)=\begin{cases}
k^{-l} \frac{(k-1)!}{(|T|-1)!} &\text{ if } |T|\leq l\\
0 &\text{ else}
\end{cases}$$
Fix $A\subset X$ and evaluate $\sum_{T: A\in T\in \mathcal{B}(X)} g(T)$. Consider ordered partitions of $X\setminus A$ into $k-1$ potentially empty sets. Each of these contributes exactly $k^{-l}$ to this sum.

To see this, consider a particular ordered partition of $X\setminus A$ into $k-1$ potentially empty sets, with $m$ non-empty sets. This corresponds to a partition $T$ of $X\setminus A$ into $m$ parts, which has weight $g(T)=k^{-l} \frac{(k-1)!}{(m-1)!}$.
Note that to find ordered partitions into $k-1$ potentially empty sets corresponding to $T$, we need to pick the locations of the $m$ sets among the $k-1$ options. There are $\frac{(k-1)!}{(m-1)!}$ ways to do this. Hence, every ordered partition of $X\setminus A$ into $k-1$ potentially empty sets contributes weight exactly $k^{-l}$ to the sum. 

Noting that there are exactly $(k-1)^{l-|A|}$ ordered partitions of $X\setminus A$ into $k-1$ potentially empty sets, we can evaluate the sum as
$$\sum_{T: A\in T\in \mathcal{B}(X)} g(T)=(k-1)^{l-|A|}k^{-l}=\left(\frac1k\right)^{|A|}\left(\frac{k-1}{k}\right)^{l-|A|}= p^{|A|}(1-p)^{l-|A|}$$
\end{proof}

\begin{proof}[Proof of \Cref{1/klem}]
Let $g:\mathcal{B}(X)\to [0,1]$ as in \Cref{1/kauxlem} and set $\mathcal{B}'(X)=\mathcal{B}(X)\setminus\{X\}$. Note that $\sum_{T\in\mathcal{B}'(X)} g(T)=p(1-p^{l-1})$. Consider any linear order on $\mathcal{B}'(X)$ and let 
$$f':\mathcal{B}'(X)\to \mathcal{P}((p/2-p^l/2-\epsilon)n), T\mapsto \left\{\left\lfloor\sum_{T'<T}  g(T')\right\rfloor \frac{(1-\eta)n}{2}+1,\dots,\left\lfloor\sum_{T'\leq T}  g(T')\right\rfloor\frac{(1-\eta)n}{2}\right\}$$
where $\eta$ is such that $\left\lfloor\sum_{T\in \mathcal{B}'(X)}  g(T)\right\rfloor\frac{(1-\eta)n}{2}=(p/2-p^l/2-\epsilon)n$.
Let
$$f:\mathcal{P}(X)\to\mathcal{P}((p/2-p^l/2-\epsilon)n);X'\mapsto \bigcup_{T\supset X'}f'(T)$$
Hence; 
\begin{align*}
\bigcup_{X':x\in X'\subset X} f(X')&=\bigcup_{X':x\in X'\subset X}\bigcup_{T:T\supset X'}f'(T)\\
&=\bigcup_{T\in \mathcal{B}(X)}f'(T)
\end{align*}

\end{proof}

\begin{proof}[Proof of \Cref{bigproposition}]
Fix $X\subset V$ with $|X|=l$. Sample all edges incident to $X$. For $I\subset X$, let $X_I=\{v\in V\setminus X: N(v)\cap X=I\}$.  Note that whp $|X_I|\geq p^{|I|}(1-p)^{l-|I|}(1-\eta/10)n$ for any $\xi>0$. Use \Cref{1/klem} to find $Y_I=f(I)$, such that $|Y_I|\leq \frac{(1-\eta/10)|X_I|}{2}$. Now sample all the other edges in the graph. By \Cref{genbob}, whp Bob can guarantee that at the end of round one the colours in $Y_I$ appears in $X_I$. By construction of $Y_I$, all vertices in $X$ will see all colours and, moreover, there is no single vertex whose recolouring changes that. Regardless of Alice' first move in the second round, Bob can choose a vertex that sees all colours in his first move in the second round. Thus, the proposition follows.
\end{proof}

Note that the condition $p=1/k$ is essential. In fact, if $p\not\in\{\frac12,\frac13,\dots\}$, then whp for any three vertices $v_1,v_2$ and $v_3$, it is impossible to assign colours $Y_I$ to $X_I=\{v:N(v)\cap \{v_1,v_2,v_3\}=I\}$, such that $|Y_I|\leq (1/2+\eta) |X_I|$ and $\bigcup_{I:i\in I\subset [3]} Y_I=[(p/2-\epsilon)n]$ for every $i\in[3]$. Crucially, \Cref{1/kauxlem} fails to hold. Hence, given that Alice can play in roughly half the vertices in $N(v)$ for all $v\in V$, at most two vertices at the end of every round can see all colours. Some must appear uniquely in the intersection, so Alice can recolour one of these in the first move of the second round. Hence, we cannot expect Bob to win at the beginning of round two. However, it is not immediately clear whether Bob cannot reintroduce the colour successfully. We believe this cannot be done, so we conjecture that 
for all $ p\in(0,1)\setminus \{\frac{1}{2},\frac{1}{3},\dots\},\exists \epsilon>0$ such that whp $\chi_g^{\infty}(G_{n,p}) \leq (1-\epsilon) \frac{pn}{2}$, as stated in \Cref{conjecture}.

\section{Answer to a question of Klostermeyer and Mendoza}

We conclude the paper by answering a question posed by Klostermeyer and Mendoza in their original paper.

They define other variants of the eternal chromatic game on graph. One of them is greedy colouring game, where Bob must colour whatever vertex he chooses with the smallest colour possible. Let $\chi_g^{\infty2} (G)$ be the smallest number $k$ such that when this game is played with $k$ colours on $G$, Alice is guaranteed to win. Further, they consider the variant of game when not only Bob, but also Alice, must use the smallest colour available for each vertex she chooses, and define $\chi_g^{\infty3} (G)$ to be eternal number of the game played with these rules. Note that clearly $\chi_g^{\infty2} (G) \leq \chi_g^{\infty3} (G)$ since Alice can, if she wishes so, choose the smallest colour for each vertex she chooses in any variant of the game and analogously $\chi_g^{\infty2} (G) \leq \chi_g^{\infty} (G)$.

Klostermeyer and Mendoza pose the following question about these new variants of the game.
 
\begin{question}
 Let $G$ be a graph with subgraph or induced subgraph $H$. Is it necessarily true that $\chi_g^{\infty 2}(G) \geq \chi_g^{\infty 2}(H)$ ? Is it necessarily true that $\chi_g^{\infty 3}(G) \geq \chi_g^{\infty 3}(H)$?
\end{question}
Indeed, it is not true. Consider the following example. 

\begin{prop}
For $n\geq2$,
$\chi_g^{\infty3}(K_{1,2n+1})=3$ and $\chi_g^{\infty2}(K_{1,2n})\geq 4$.
\end{prop}
\begin{proof}
For $\chi_g^{\infty3}(K_{1,2n+1})=3$ note that Alice starts every round as the number of vertices is even. Every round she will first play in the central vertex which will become the unique element from $[3]$ not yet appearing in the graph. All the other vertices will now become the former colour of the central vertex.

For $\chi_g^{\infty3}(K_{1,2n})\geq 4$. assume for a contradiction 3 colours suffice and note that Bob begins the second round. Let $x$ for the central vertex. Then $x$ is either adjacent to two different colours or $N(x)$ is monochromatic. In the former case, Bob plays in $x$ and finds that there is no colour available, a contradiction.\\
In the latter case, Bob plays in $N(x)$, bringing the number of colours in $N(x)$ to two. Hence, Alice cannot play in $x$. She can also not bring down the number of colours in $N(x)$ as it contains at least three vertices. Thence, when Bob gets to play his second move in the second round, and plays $x$, he finds no colours available, again a contradiction. 
\end{proof}

Note that $H=K_{1,2n}$ is an (induced) subgraph of $G=K_{1,2n+1}$, and 
$$\chi_g^{\infty2}(G)\leq\chi_g^{\infty3}(G)\leq\chi_g^{\infty2}(H)\leq\chi_g^{\infty3}(H)$$
This answers all the subquestions in the negative.

Finally, note that while there is no clear relationship between $\chi_g^{\infty3} (G)$ and $\chi_g^{\infty} (G)$ for general graphs $G$, in our definition of strategy of Alice in section 2, we let her always play the smallest colour available, and so in particular we show $\chi_g^{\infty3} (G_{n,p}) \leq  (\frac{p}{2}+o(1))n$ whp.

\section*{Acknowledgements}
The authors are grateful to their PhD  supervisor B\'ela Bollob\'as for his useful comments and support.

\bibliographystyle{abbrv}
\bibliography{references}
 
\appendix
\section{Proofs of structural results in random graph} \label{structure}

In this appendix, we provide proofs of various structural results about $G_{n,p}$ that were used in earlier proofs. Some of them will be shown in a more general form. One of our main tools will be the following well-known form of Hoeffding's Inequality.

\begin{lemma}\label{hoeffding}
For any $\epsilon>0, n \in \mathbb{N},$ and $ p \in (0,1)$, $\PP(\text{Bin}(n,p) \geq (p+\epsilon)n) \leq \exp(-2 \epsilon^2 n)$ and $\PP(\text{Bin}(n,p) \leq (p-\epsilon)n) \leq \exp(-2 \epsilon^2 n)$.
\end{lemma}

Note that Hoeffding's inequality implies \ref{vertexcondn} from Section 2 and \ref{nodegreesmall} from Section 3.





To prove \ref{fewcolours} and \ref{othercolours} from Section 2, we first prove the following result.

\begin{lemma} \label{fewalmostfull}
For all $\alpha>0$, $p \in (0,1)$, there exist constants $ K=K(\alpha,p),\beta=\beta(\alpha,p)>0$ such that whp the following holds. For any colouring of $G_{n,p}$ with $\alpha n$ colours, number of vertices that have all but at most $\beta n$ colours in their neighbourhood is at most $K \log n$.
\end{lemma}

\begin{proof}
Let $q=1-(1-p)^{2/ \alpha}$, $\beta=\frac{\alpha (1-q)}{8}$, and $ K=\frac{4}{(1-q)^2}$. 

Note that in any colouring of $G_{n,p}$ by $\alpha n$ colours, we have $\frac{\alpha}{2}n$ colours appearing at most $\frac{2}{\alpha}$ times each. 

Assume there exists a set $S$ of $K \log n$ vertices missing at most $\beta n$ colours each. For $n$ satisfying $K \log n < \frac{\alpha}{4} n$), there exists a set $C$ of $\frac{\alpha}{4}n$ colours appearing at most $\frac{2}{\alpha}$ times each such that no vertex in $S$ has any colour from $C$. In particular, there must be mutually disjoint sets of vertices $S,T_{1},...,T_{\frac{\alpha}{4}n}$, such that $|T_{i}| \leq \frac{2}{\alpha}$ for each $i$, $|S|=K \log n$ and each vertex in $S$ is joined to at least $(\frac{\alpha}{4}-\beta)n$ sets $T_{i}$ in our graph.

Now, we consider the probability that such structure exists in $G_{n,p}$. For $n$ sufficiently large, we find there are $\sum_{i=1}^{2/ \alpha} {n \choose i} \leq 2 {n \choose 2/ \alpha}$ ways of choosing each of the sets $T_i$. So for such large $n$, we have at most

\begin{align*}
{n \choose K \log n} \left(2 {n \choose 2/ \alpha}\right)^{\frac{\alpha n}{4}} &\leq n^{K \log n} 2^{\frac{\alpha n}{4}} \left(\frac{e \alpha n}{2}\right)^{\frac{\alpha n}{4}}\\ 
&=  \exp\left(K (\log n)^2+ \frac{\alpha n}{4} \log n + \frac{\alpha}{4}n \log \left(e \alpha\right) \right)
\end{align*}
ways to choose sets $S,T_{1},...,T_{\alpha n/4} $. Now for any such fixed choice, the probability that these sets satisfy the conditions is at most

\begin{align*}
\PP\left(\text{Bin}\left(\frac{\alpha n}{4},q\right) \geq \left(q+\frac{1-q}{2}\right)\frac{\alpha n}{4} \right)^{K \log n} \leq \exp \left(-2\left(\frac{(1-q)}{2}\right)^2 \frac{\alpha n}{4} K \log n\right)
\end{align*}

The union bound then gives that the probability of finding appropriate $S,T_1,\dots,T_{\alpha n/4}$
\begin{align*}
\exp\Bigg(K (\log n)^2+ &\frac{\alpha n}{4} \log n + \frac{\alpha}{4}n \log \left(e \alpha\right)-2\left(\frac{(1-q)}{2}\right)^2 \frac{\alpha n}{4} K \log n \Bigg) \\
&=\exp\left(K (\log n)^2 + \frac{\alpha}{4}n \log \left(e \alpha\right)+\left(1-K\frac{(1-q)^2}{2}\right) \frac{\alpha n}{4}  \log n \right)\\
&=\exp\left(K (\log n)^2 + \frac{\alpha}{4}n \log \left(e \alpha\right)- \frac{\alpha n}{4}  \log n \right)=o(1)
\end{align*}
The result follows.
\end{proof}

To conclude \ref{fewcolours}, simply plug in $\alpha=(\frac{p}{2}+ \epsilon)$, and note that if some value of $\beta>0$ works, then any smaller one does too, so we can insist on $\beta$ being not too large.

To conclude \ref{othercolours}, plug in $\alpha=\frac{\epsilon}{200}$ and note that presence of other colours only helps us, as the result would still hold even if all the other vertices were also coloured in $\frac{\epsilon}{200} n$ small colours.

The following implies \ref{fewdangers} from section 2.

\begin{lemma}
Fix any $\epsilon,$ and $\delta$ greater than $0$. Assume $K \in \N$ is fixed, such that $K>\frac{6\epsilon}{\delta^2}$. 
Whp, if $A,B\subset V(G)$ are disjoint subsets with $|A|=|B| \geq  \epsilon n$, then there are less than $K$ vertices connected to at least $\delta n$ more vertices in $B$ than in $A$.
\end{lemma}

\begin{proof}
Assume for a contradiction $n\geq \frac{2K}{\epsilon}$ and there exist $A,B,$ as stated in the lemma such that there are at least $K$ vertices connected to at least $\delta n$ more vertices in $B$ than in $A$. Let $S$ be a collection of $K$ such vertices.
Let $A'=A\setminus S$ and $B'=B\setminus S$. Note that $e(A',S)\leq e(A,S)$ and $e(B',S)\geq e(B,S)-K^2$, so that 
$$e(B',S)-e(A',S)\geq e(B,S)-e(A,S)-K^2\geq K\delta n-K^2\geq K\delta n/2.$$ Hence, either $e(B',S)\geq (|B'|\cdot |S| p)+\delta Kn/8$ or $e(A',S)\leq (|B'|\cdot |S| p)-3\delta Kn/8\leq (|A'|\cdot |S| p)-\delta Kn/8$.
The probability of the former (the latter follows analogously) is given by;
\begin{align*}
\PP\left(\text{Bin}(|B'|\cdot |S|, p)\geq (|B'| K p)+\delta Kn/8\right)&\leq \exp\left(-2 \left(\frac{\delta n}{8|B'|}\right)^2|B'| K\right)\\
&\leq\exp\left(- \frac{\delta^2 nK}{2\epsilon }\right)
\end{align*}



We can choose sets $A,B,$ and $S$ in at most

\begin{align*}
{n \choose |A|}^2  {n \choose K} \leq 2^{3n} 
\end{align*}
ways. Thus, the probability that any such sets $A,B$ and $S$ exist is at most 

\begin{align*}
2 \exp\left(-\frac{\delta^2}{2 \epsilon}nK + 3 n \log 2\right)  \rightarrow 0  
\end{align*}
provided $K>\frac{6\epsilon}{\delta^2}>\frac{6\epsilon}{\delta^2} \log 2$.
\end{proof}

The proof of \ref{fewblocks} from section 3 follows from the fact that for $\delta(p)$ sufficiently small and positive, whp there exists no three vertices $u,v,w$ such that the number of vertices in the neighbourhood of $u$, but not in the neighbourhood of $v$ or $w$ is at most $\delta n$ by Hoeffding's Inequality. \ref{fewblocks} follows directly from the following setting $m=2$.

\Cref{generalfewblocks} follows in the same manner, this time using the particular $m$ we need.





\begin{lemma}
Fix $m\in\mathbb{N}$, $\gamma>0$, $p \in (0,1)$. Then for any $K>\frac{4}{\gamma(1-p)^{2m}}$, whp $G_{n,p}$ does not contain any collection of sets $S,T_{1},...,T_{K}$ such that $T_{1},...,T_{K}$ are all mutually disjoint, $|S| \geq \gamma n$, $|T_{1}|=...=|T_{K}|=m$ and for every $T_{i}$, all but at most $\frac{(1-p)^m}{4} \gamma n$ vertices of $S$ are connected to at least one vertex in $T_{i}$.
\end{lemma}
\begin{proof}
Provided $n>\frac{2Km}{\gamma}$ , we can find $S' \subset S$ such that $|S'|=\frac{\gamma n}{2} $ and $S'$ is disjoint from all of $T_{1},...,T_{K}$. For any such fixed $S',T_{1},...,T_{K}$, by Hoeffding's Inequality, the probability that for each $T_{i}$, all but at most $\frac{(1-p)^m}{4} \gamma n$ vertices of $S'$ are connected to at least one vertex in $T_{i}$ is at most

\begin{align*}
\P\left(\text{Bin}\left(\frac{\gamma}{2}n,(1-p)^m\right) \leq \frac{(1-p)^m}{2}\gamma n - \frac{(1-p)^m}{4}\gamma n\right)^K
& \leq \exp\left(-2 \frac{(1-p)^{2m}}{4} \frac{\gamma}{2}nK\right) \\
&= \exp\left(-\frac{(1-p)^{2m} \gamma Kn}{4}. \right)
\end{align*}

There are at most
\begin{align*}
{n \choose  m}^K  {n \choose \gamma n/2} &\leq n^{mK} 2^n\\
&=\exp\left(mK \log n + n\log(2) \right)
\end{align*}

ways to choose such sets $S',T_{1},...,T_{K}$. So, as long as $K> \frac{4}{\gamma(1-p)^{2m}}> \frac{4}{\gamma(1-p)^{2m}}\log(2)$, the result follows.
\end{proof}
\end{document}